\documentclass[reqno,11pt,final]{amsart}

\usepackage[utf8]{inputenc}
\usepackage[russian,english]{babel}
\usepackage{csquotes} 
\usepackage[a4paper,twoside]{geometry}
\usepackage{graphicx}
\usepackage[backend=bibtex,giveninits=true, doi=false, style=alphabetic,bibencoding=utf8,language=auto,autolang=other,maxnames=10, maxalphanames=10]{biblatex}
\usepackage{caption} 
\usepackage[hyperfootnotes=false]{hyperref} 
\hypersetup{
	colorlinks = true,
	linkcolor = {blue},
	urlcolor = {red},
	citecolor = {blue}
}
\usepackage{amsmath}
\usepackage{amsthm}
\usepackage{amssymb}
\usepackage{esint} 
\usepackage{mathrsfs} 
\usepackage{faktor} 
\usepackage{dsfont} 
\usepackage{xcolor}
\usepackage{array}
\usepackage{hhline}
\usepackage{enumitem} 
\usepackage{comment} 
\usepackage[toc,page]{appendix} 
\usepackage{imakeidx} 
\usepackage{xparse} 
\usepackage{mathtools}
\usepackage{fancyhdr} 
\usepackage{ifthen} 
\usepackage{forloop} 
\usepackage{xstring}
\usepackage{tikz}
\usetikzlibrary{babel} 
\usetikzlibrary{cd} 
\usepackage{emptypage} 

\usepackage[nameinlink,capitalise]{cleveref} 
\crefname{equation}{}{} 

\newcounter{results}[section] 

\theoremstyle{plain}
\newtheorem{theorem}[results]{Theorem}
\newtheorem{lemma}[results]{Lemma}
\newtheorem{proposition}[results]{Proposition}

\newtheorem*{theorem*}{Theorem}
\newtheorem*{lemma*}{Lemma}
\newtheorem*{proposition*}{Proposition}
\newtheorem*{corollary*}{corollary}
\newtheorem*{exercise*}{Exercise}
\newtheorem*{fact*}{Fact}
\newtheorem*{problem*}{Problem}

\theoremstyle{remark}
\newtheorem{remark}[results]{Remark}

\newtheorem*{remark*}{Remark}
\newtheorem*{question*}{Question}

\theoremstyle{definition}
\newtheorem{definition}[results]{Definition}

\newtheorem*{definition*}{Definition}
\newtheorem*{example*}{Example}

\numberwithin{equation}{section}

\newcommand{\Z}{\ensuremath{\mathbb Z}} 
\newcommand{\R}{\ensuremath{\mathbb R}} 

\newcommand{\scalprod}[2]{\ensuremath{\langle #1, #2\rangle}} 
\newcommand{\Haus}{\ensuremath{\mathscr H}} 

\newcommand{\de}{\ensuremath{\,\mathrm d}}

\DeclareMathOperator{\diam}{diam}

\newcommand{\E}[1]{\ensuremath{\operatorname{\mathbb{E}}\left[#1\right]}}
\renewcommand{\P}[1]{\ensuremath{\operatorname{\mathbb{P}}\left(#1\right)}}

\let\div\undefined
\DeclareMathOperator{\div}{div}

\ifdefined\comma 
	\renewcommand{\comma}{\ensuremath{\, \text{, }}}
\else 
	\newcommand{\comma}{\ensuremath{\, \text{, }}}
\fi

\newcommand{\s}{\mbox{ }}

\newcommand{\mm}{\mathfrak{m}}
\newcommand{\T}{{\mathbb{T}^2}}

\makeatletter
\newcommand{\newreptheorem}[2]{%
  \newtheorem*{rep@#1}{\rep@title}%
  \newenvironment{rep#1}[1]%
    {\def\rep@title{#2 \ref*{##1}}\begin{rep@#1}}%
    {\end{rep@#1}}%
}%
\makeatother
\theoremstyle{plain}
\newreptheorem{theorem}{Theorem}

\renewbibmacro{in:}{} 
\DeclareFieldFormat[article]{pages}{#1} 
\DeclareFieldFormat[book]{pages}{#1}
\DeclareFieldFormat[article]{title}{#1} 
\DeclareFieldFormat[book]{title}{\textit{#1}}

\renewbibmacro*{volume+number+eid}{ 
  \printfield{volume}
  \setunit*{\addcomma \addnbspace}
  \printfield{number}
  \setunit{\addcomma\space}
  \printfield{eid}}
  
\renewbibmacro*{publisher+location+date}{
  \printlist{publisher}
  \iflistundef{location}
    {\setunit*{\addcomma\space}}
    {\setunit*{\addcolon\space}}
  \printlist{location}
  \setunit*{\addcomma\space}
  \usebibmacro{date}
  \newunit}
  
  \title[PDE estimates for random matching with power cost]{Sharp PDE estimates for random two-dimensional bipartite matching with power cost function}
\author{L. Ambrosio}
\address{Scuola Normale Superiore, Pisa}
\email{luigi.ambrosio@sns.it}
\author{F. Vitillaro}
\address{Scuola Normale Superiore, Pisa}
\email{federico.vitillaro@sns.it}
\author{D. Trevisan}
\address{Universit\`a di Pisa}
\email{dario.trevisan@unipi.it}

\date{\today}
  
\begin{document}

\maketitle

\begin{abstract}
We investigate the random bipartite optimal matching problem on a flat torus in two-dimensions, considering general strictly convex power costs of the distance. We extend the successful ansatz first introduced by Caracciolo et al. for the quadratic case, involving a linear Poisson equation, to a non-linear equation of $q$-Poisson type, allowing for a more comprehensive analysis of the optimal transport cost. Our results establish new asymptotic connections between the energy of the solution to the PDE and the optimal transport cost, providing insights on their asymptotic behavior.
\end{abstract}

\section{Introduction and main result}

Let $(X_1, \ldots, X_n)$, $(Y_1, \ldots, Y_n)$ be two sets of $n$ random points independent and uniformly distributed on the flat torus $\T=\R^2/\Z^2$, i.e., with common law given by the Lebesgue measure $\mm$  on $\T$. 
The random bipartite optimal matching problem concerns the study of the optimal coupling (with respect to a certain cost function) of these points, that is,  the optimal transport  from the empirical measure $\mu^n:=\frac 1n \sum_{i=1}^n \delta_{X_i}$ to $\nu_n := \frac 1 n \sum_{i=1}^n \delta_{Y_i}$, in particular in the limit $n \gg 1$.

For the quadratic cost $c(x,y):=\mathbf{d}(x,y)^2$, where $\mathbf{d}$ is the quotient (flat) distance in $\T$, the seminal paper \cite{CLPS} gave a very appealing PDE ansatz on the asymptotic of the expectation of the optimal transport cost, based on a linearization of the Monge-Ampère equation. While it was already known in the literature that, for the cost $c=\mathbf{d}^p$ in dimension $d=2$, the expectation of the optimal transport cost behaves like $(n^{-1} \ln n )^{p/2}$ (for $p=1$ since \cite{AKT}), in \cite{CLPS} they managed to predict the limit coefficient as $1/(2\pi)$ in the case $p=2$, exploiting Fourier analysis and some renormalization procedure. This prediction was then rigorously proven in \cite{AST}, together with a new PDE proof of the classical bounds in \cite{AKT}.

Since then, several works have been using such PDE ansatz to estimate with different degrees of sharpness the asymptotics of random optimal matching costs and their solutions, in several settings. Focusing only on the two-dimensional case, but possibly including more general manifolds than $\T$, we mention here the rigorous results \cite{BobLe19, AG, AGT19, goldman2021quantitative,  ambrosio2022quadratic, goldman2022fluctuation, huesmann2023there, goldman2023almost, clozeau2024annealed} as well as further intriguing predictions from the physical literature \cite{benedetto2020euclidean, benedetto2021random} and refer e.g.\ to the  contribution \cite{ledoux2022optimal} for a more general overview on the subject.

The aim of the present work is to establish new asymptotic connections between the solution of a ``linearized PDE'' and the expectation of the
optimal transport cost, on $\T$, for general $p>1$,  extending the main results in \cite{AST, AG}. Let us mention here that recently other works  focused on two-dimensional random optimal matching problems, beyond the quadratic cost, in particular \cite{koch2024geometric}, where the  quantitative harmonic approximation techniques -- originally in \cite{goldman2021quantitative}, see also the exposition \cite{koch2023lecture} -- are extended to any $p>1$, and the preprint \cite{huesmann2023stationary}, where existence of a $p$-cyclically monotone stationary matching from a Poisson point process to the Lebesgue measure is ruled out for any $p>1$ -- the quadratic case is covered in \cite{huesmann2023there}.

In order to describe here informally our results, we may treat the empirical measures $\mu^n = \rho_0 \mm$, $\nu^n=\rho_1 \mm$ as absolutely continuous with respect to $\mm$. This will be made rigorous by a regularization with the heat kernel $P_t$ on $\T$, as performed in \cite{AST}. We first recall (see e.g., \cite[Remark 5.3]{OT}) that the Kantorovich potential $\phi$ is related to the optimal transport map $T$ by the identity
$$ T(x)=x-|\nabla \phi(x)|^{q-2}\nabla \phi(x),$$
where, throughout the paper, $q= p/(p-1)$ denotes the dual exponent of $p$. Then, the Monge-Ampère equation takes the form
$$\rho_1(x-|\nabla \phi (x)|^{q-2}\nabla \phi(x))\det (\nabla(x-|\nabla \phi (x)|^{q-2}\nabla \phi (x)))=\rho_0(x).$$
This PDE contains three non-linearities: the determinant, the dependence of $\rho_1$ on $\nabla\phi$ and finally, when $p\neq 2$, the
nonlinear term $|\nabla\phi|^{q-2}\nabla \phi$. Our main result shows that in order to obtain a good first-order approximation of the expected value of the transport
cost it is sufficient to remove only the first two non-linearities, keeping the third one. This invokes the ``linearized'' (but still nonlinear!) PDE of $q$-Poisson type
 \begin{equation}\label{Lin-PDE}
-{\rm div}\bigl(|\nabla\phi|^{q-2}\nabla\phi)=\rho_1-\rho_0,\qquad\phi\in H^{1,q}(\T)
\end{equation}
in the sense of distributions, namely
\begin{equation}\label{q-Poisson}\int_\T |\nabla\phi|^{q-2} \langle \nabla\phi, \nabla \eta \rangle \s d\mm=\int (\rho_1-\rho_0)\eta \s d\mm
\qquad\forall\eta \in H^{1,q}(\T),
\end{equation}
 where we always assume, to ensure uniqueness, that $\int_\T\phi\de\mm=0$, yielding the approximation
 \begin{equation}
  |T(x) - x|^p \approx | |\nabla \phi(x)|^{q-2} \nabla \phi(x) |^p = |\nabla \phi(x)|^q.
 \end{equation}


Our main result makes precise such approximation (see \cref{sec:notation} for more details on the notation).

\begin{theorem}[Main result]\label{tmain}
If $(X_i)_{i=1}^\infty$ and $(Y_i)_{i=1}^\infty$ are independent and identically distributed random variables with law $\mm$ on $\T$, then
$$
\lim_{n\to\infty}\biggl(\frac{n}{\ln n}\biggr)^{p/2}\biggl\vert \E{W_p^p(\mu^n,\nu^n)}-\E{\int_\T|\nabla\phi_n|^q\de\mm}\biggr\vert=0
$$
where
\begin{equation}\label{eq:empirical}
\mu^n=\frac 1n\sum_{i=1}^n\delta_{X_i},\qquad\nu^n=\frac 1n\sum_{i=1}^n\delta_{Y_i},
\end{equation}
and $\phi_n$ is the solution to \eqref{q-Poisson} with random right hand side $\rho_1-\rho_0 = \rho_{1,n} - \rho_{0,n}$ and
\begin{equation}\label{eq:regularizedrhoi}
\rho_{0,n}\mm=P_{t_n}\mu^n,\qquad\rho_{1,n}\mm=P_{t_n}\nu^n
\end{equation}
provided $t_n\gg n^{-1}\ln n$ and $\ln (nt_n)\ll\ln n$.
\end{theorem}

For instance, a good choice of the intermediate regularization scale $t_n$ in the main result would be $t_n=n^{-1}(\ln n)^\beta$ with $\beta>1$. Thanks to this result, the existence of the limit
$$
\lim_{n\to\infty}\frac{\E{W_p^p(\mu^n,\nu^n)}}{\bigl((\ln n)/n\bigr)^{p/2}}
$$
is equivalent to the existence of the limit when, in the numerator, $\E{W_p^p(\mu^n,\nu^n)}$ is replaced by
$\E{\int_\T|\nabla\phi_n|^q\de\mm}$, with $\phi_n$ solutions to the PDE \eqref{q-Poisson} with a random right hand side \eqref{eq:regularizedrhoi}.
It would be interesting to prove or disprove the existence of the limit thanks to this reduction to a stochastic PDE.


In order to prove \cref{tmain}, the only probabilistic ingredients (see \cref{sec:heat}) will consist in checking that as $n \rightarrow \infty$, with high probability the densities $\rho_{i,n}$ in \eqref{eq:regularizedrhoi}, for $i=0,1$ are both sufficiently close to the constant density (\cref{prop_Bernstein}), as well as not too far from $\mu^n$ and $\nu^n$ in the Wasserstein sense (\cref{prop:dispersive}), collecting and slightly extending some results from \cite{AG} and \cite{AST}. Then, in \cref{UB} and \cref{LB} we will focus our efforts on showing the following deterministic result.

\begin{theorem}\label{ULB}
Let $p>1$, let $\phi$ be a solution of \eqref{q-Poisson}, and let $c:=2\max\limits_{i=0,1}\|\rho_i-1\|_{L^\infty(\T)}$. Then there exist $\underline{\delta}=\underline{\delta}(c,p)$ and $\overline{\delta}=\overline{\delta}(c,p)$ such that
$\underline{\delta}+\overline{\delta} \rightarrow 0$  as $c \rightarrow 0$ and
$$(1-\underline{\delta}) \int_\T |\nabla \phi|^q \de\mm \leq W_p^p(\rho_0 \mm, \rho_1 \mm) \leq (1+\overline{\delta}) \int_\T |\nabla \phi|^q \de\mm.$$
\end{theorem}

This result actually holds, with the same proof, on any $d$-dimensional torus. 
The extension to the setting of compact Riemannian manifolds
(along the lines of \cite{AST}) possibly with boundary is beyond the scope of this note, and requires in particular the understanding in that more general setting of the stability of the estimates from above for the Riemannian analogous of the operator ${\rm div}(|\nabla\phi|^{q-1}\nabla\phi)$ under the action of the Hopf-Lax semigroup, even after shocks.

\smallskip
{\bf Acknowledgements.} The authors thank N. Gigli for having pointed out to them  \cite[Section~4.3]{Ohta} where, on Riemannian manifolds, via the theory of characteristics,  the preservation of upper bounds on the $q$-Laplacian under the action of the $p$-Hopf-Lax
semigroup is shown, even in the nonlinear case $p\neq 2$, before shocks. Eventually, in the case of the flat space $\T$,
the proof we gave in \cref{Prop-viscos} does not use this computation and works even beyond shocks, using solutions in the viscosity sense.

L.A. and F.V. acknowledge the PRIN Italian grant 202244A7YL ``Gradient Flows and Non-Smooth Geometric Structures with
Applications to Optimization and Machine Learning''. D.T.\ acknowledges the MUR Excellence Department Project awarded to the Department of Mathematics, University of Pisa, CUP I57G22000700001,  the HPC Italian National Centre for HPC, Big Data and Quantum Computing - Proposal code CN1 CN00000013, CUP I53C22000690001, the PRIN Italian grant 2022WHZ5XH - ``understanding the LEarning process of QUantum Neural networks (LeQun)'', CUP J53D23003890006, the INdAM-GNAMPA project 2023 ``Teoremi Limite per Dinamiche di Discesa Gradiente Stocastica: Convergenza e Generalizzazione'', INdAM-GNAMPA project 2024 ``Tecniche analitiche e probabilistiche in informazione quantistica'' and the project  G24-202 ``Variational methods for geometric and optimal matching problems'' funded by Università Italo Francese.  Research also partly funded by PNRR - M4C2 - Investimento 1.3, Partenariato Esteso PE00000013 - "FAIR - Future Artificial Intelligence Research" - Spoke 1 "Human-centered AI", funded by the European Commission under the NextGeneration EU programme.

\section{Preliminaries}\label{sec:notation}

\subsection{The Wasserstein distance}
Given probability measures $\mu$, $\nu$  on $\T$ and $p \ge 1$ we define the $p-$Wasserstein distance  between $\mu$ and $\nu$ as
$$ W_{p}(\mu,\nu):=\left.\min\left\{ \left(\int_{\T\times\T} \mathbf{d}(x,y)^p d\pi(x,y) \right)^{1/p} \right|\pi_1=\mu,\pi_2=\nu \right\}.$$
We refer to \cite{OT} for an introduction to the subject. In particular, we will use throughout that $W_p$ enjoys the triangle inequality. Moreover, we recall here for later use the following consequence of the Benamou-Brenier formula, see e.g. \cite{peyre2018comparison}, \cite[Theorem 2]{Le17} or \cite[Lemma 3.4]{goldman2021convergence}.

\begin{proposition}\label{BBcomodo} Let $\mu = \rho_0 \mm$, $\nu=\rho_1 \mm$ be absolutely continuous with respect to $\mm$ and let $\phi$ be a solution to \eqref{q-Poisson} with $q=2$. Then, for every $p \ge 1$ there exists a constant $C = C(\T, p)< \infty$ such that
\begin{equation}\label{eq:peyre}
 W_p^p(\mu,\nu)\le C (\operatorname{ess-inf} \rho_1 )^{1-p} \int_{\T}|\nabla\phi|^p d \mm.
\end{equation}
\end{proposition}

We notice that the bound above is asymmetric in the roles of $\mu$ and $\nu$, since only $\rho_1$ is required to be (essentially) bounded from below. In some sense, our work aims to sharpen \eqref{eq:peyre} by replacing the linear Poisson equation with the non-linear $q$-Poisson one, and indeed \cref{prop:upper-bound} below is proved using a similar argument. However, \eqref{eq:peyre} is useful as one can combine it with harmonic analysis tools, as done e.g.\ in \cite{AST, Le17}. For example, for any $p>1$, by the classical boundedness of the Riesz transform operator $\nabla (-\Delta)^{-1/2}$ on $\T$, where $(-\Delta)^{-1/2}$ is defined as a Fourier multiplier, one can further bound from above
 $$ \int_{\T}|\nabla\phi|^p d \mm \le C \int_{\T} |(-\Delta)^{-1/2}  (\rho_1-\rho_0)|^p d \mm,$$
where $C = C(\T, p)<\infty$. Hence, from \eqref{eq:peyre} we further deduce the upper bound
\begin{equation}\label{eq:peyre-riesz}
  W_p^p(\mu,\nu)\le C (\operatorname{ess-inf} \rho_1)^{1-p} \int_{\T} |(-\Delta)^{-1/2} (\rho_1-\rho_0)|^p d \mm,
\end{equation}
where again $C = C(\T, p)< \infty$.
%
%

\subsection{Viscosity solutions} Viscosity solutions are designed to give a suitable notion of solution (with good properties such us uniqueness, stability and comparison principles) for general non-linear equations for which the distributional point of view does not make sense, as fully nonlinear PDE's. However, this notion reveals to be useful also for PDE's having a distributional formulation. This is the case of the $q$-Laplace (also called $q$-Poisson) equation considered in this paper, associated to the
differential operator 
$$
 -\Delta_q u:=-{\rm div}\bigl(|\nabla u|^{q-2}\nabla u).
 $$
Actually, we will just deal with supersolutions.

\begin{definition}\label{DefV}
Let $g:\T\to\R$. We say that a function $f: \T \rightarrow (-\infty, +\infty]$ is a \textbf{viscosity supersolution} for the equation $-\Delta_q u+g=0$, and we write
\begin{equation}\label{pP} -\Delta_q u+g\geq 0\quad\text{in the viscosity sense}\end{equation}
if the following conditions hold:
\begin{enumerate}[label=(\roman*)]
\item $f$ is lower semicontinuous, $f \not\equiv +\infty$, and
\item \label{visc-def} whenever $x_0\in\T$ and $\varphi\in C^2(\T)$ are such that $f-\varphi$ has a local minimum at $x_0$ and $\nabla \varphi(x_0) \neq 0$, we have
$$-\Delta_q \varphi(x_0)+g(x_0) \geq 0.$$
\end{enumerate}
\end{definition}

\cref{DefV} is adapted to the special form of the $q$-Laplace PDE. Indeed, the additional requirement $\nabla \varphi(x_0) \neq 0$
(not present in the general theory of viscosity solutions, see for instance \cite{CIL}) is due to the fact that the expression
\begin{equation}\label{expansion p-l}\Delta_q \varphi=|\nabla \varphi|^{q-4} \left[ |\nabla \varphi|^2 \Delta \varphi+ 
(q-2) \sum_{i,\,j=1}^n \frac{\partial \varphi}{\partial x_i}\frac{\partial \varphi}{\partial x_j}\frac{\partial^2 \varphi}{\partial x_i \partial x_j}\right]\end{equation}
is singular at the critical points of $\varphi$, when $1<q<2$.

\begin{remark}
With this convention, any $f\in C^2(\T)$ satisfying  $-\Delta_q f+g\geq 0$ in the pointwise sense is also a viscosity supersolution. This follows from the fact that if we call 
$$F_q(v, S):(\R^2\setminus\{0\})\times \mbox{Sym}^{2 \times 2}(\R)\to\R$$ 
the differential operator such that $F_q(\nabla u, \nabla^2 u)=-\Delta_q u$, then $F$ is non-increasing with respect to $S$ (just look at (\ref{expansion p-l})). It follows that if 
$f-\varphi$ has a local minimum at $x_0$ with $\nabla \varphi(x_0)\neq 0$, then $\nabla f(x_0)=\nabla \varphi(x_0)\neq 0$ and 
$$F_q(\nabla \varphi(x_0), \nabla^2 \varphi(x_0))+g(x_0)\geq F_q(\nabla f(x_0), \nabla^2 f(x_0))+g(x_0)\geq 0$$
as $\nabla^2 f(x_0)\geq\nabla^2\varphi(x_0)$.
\end{remark}

\subsection{Hopf-Lax semigroup}\label{UB} Given $f:\T\to\R$ lower semicontinuous, let $u=Q_tf$ be the \textit{Hopf-Lax semigroup} associated to the Hamilton-Jacobi equation
\begin{equation}\label{HJ}
\partial_t u + \frac{|\nabla u|^q}q=0, 
\end{equation}
that is
\begin{equation}\label{eq:defHL}
(Q_tf)(x)=\min_{y \in \T} \left\{f(y)+ \frac{{\mathbf d}^p(x,y)}{pt^{p-1}} \right\}.
\end{equation}
The following properties of the semigroup $Q_tf$, with $Q_0f=f$, are well-known, see for instance Proposition 3.3 in \cite{AGS} for a detailed proof.

\begin{proposition}\label{prop_HL} Let $f:\T\to\R$ be Lipschitz. Then the functions $Q_tf$ are Lipschitz, uniformly with respect to
$t\in [0,1]$, $t\mapsto Q_tf$ is Lipschitz from $[0,1]$ to $C(\T)$ and the PDE \eqref{HJ} is satisfied almost everywhere in $(0,1)\times\T$. 
\end{proposition}

\subsection{Heat kernel on $\T$}\label{sec:heat}

We recall that the heat kernel on the torus $\T=\R^2/\Z^2$ is given by
\begin{equation}\label{def p_t}
p_t(x):=\sum_{\mathbf{n} \in \Z^2} \overline{p}_t(x+\mathbf{n}),
\end{equation}
where $\overline{p}_t(x)=\frac1{4\pi t} e^{-\frac{|x|^2}{4t}}$, $x\in\R^2$, is the Euclidean heat kernel.
Given a probability measure $\mu$ in $\T$, we denote by $P_t\mu\ll\mm$ the probability measure having density
$$
\rho(x)=\int_\T p_t(x-y)\de\mu(y).
$$
Let us recall that $(P_t)_{t \ge 0}$ defines a symmetric Markov (convolution) semigroup, $P_{s+t} = P_s \circ P_t$ with (unique) invariant measures $\mm$ and generator given by the (distributional) Laplacian.  Let  us recall the following deterministic dispersion bound, directly coming from the coupling
$$
\Sigma=\int_\T p_t(z)\Sigma_z\de\mm(z)\qquad\text{with}\qquad\Sigma_z=(Id\times\tau_z)_\#\mu
$$
between $\mu$ and $P_t\mu$ (where $\tau_z$ is the shift map):
\begin{equation}\label{eq:dispersion}
W_p(\mu,P_t\mu)\leq C_0\sqrt{t}
\qquad\forall t>0,
\end{equation}
with $C_0=C_0(\T)=\bigl(\int_{\T}|z|^pp_1(z)\de\mm(z)\bigr)^{1/p}$ for any probability measure $\mu$ in $\T$.

A remarkable fact, first noticed in \cite[Theorem~5.2]{AG}, is that the dispersion bound  above can be significantly improved (in average) when applied to empirical measures $\mu^n$ as in \eqref{eq:empirical}.

%
%
%
\begin{proposition} \label{prop:dispersive} For every  $p \ge 1$ there exists positive constant $C_1(\T,p),\,C_2(\T,p)$ such that the following holds. If 
$t=\alpha/n\leq \frac 12$ with $\alpha\geq C_1(\T,p)\ln n$, then
$$
\E{W_p^p(\mu^n,P_t\mu^n)}\leq C_2(\T,p)\left( \frac{\ln\alpha}{n} \right)^{p/2}.
$$
\end{proposition}

\begin{proof}
The case $p=2$ is established in \cite[Theorem~5.2]{AG}, and by the H\"older inequality, it entails the thesis for every $1 \le p <2$:
\begin{equation}\label{eq:ag_p}
\E{W_p^p(\mu^n,P_t\mu^n)}\leq (C_2)^{p/2}\frac{(\ln\alpha)^{p/2}}{n^{p/2}},
\quad t=\frac{\alpha}{n},\,\,\alpha\geq C_1\ln n.
\end{equation}
Therefore, it is sufficient to consider the case $p \ge 2$. To this aim, we combine the argument from \cite{AG} with the application of Rosenthal's inequality, from \cite{Le17}, where upper bounds for the random bipartite matching cost are proved for any $p \ge 2$. By the triangle inequality and the elementary bound
\begin{equation}\label{eq:trivial-p}
 |x+y|^p \le 2^{p-1} ( |x|^p+ |y|^p)
\end{equation}
for some $C = C(p)<\infty$, we find
\begin{equation}\begin{split}
 \E{W_p^p(\mu^n,P_t\mu^n)} & \le 2^{p-1} \left( \E{W_p^p(\mu^n,P_{1/n}\mu^n)} +   \E{W_p^p(P_{1/n}\mu^n, P_{t}\mu^n )} \right)\\
 & \le  2^{p-1} \left( C_0 n^{-p/2} +   \E{W_p^p(P_{1/n}\mu^n, P_{t}\mu^n )} \right),
 \end{split}
\end{equation}
having used \eqref{eq:dispersion} in the second inequality. Thus, we are reduced to bound from above the expectation of $W_p^p(P_{1/n}\mu^n, P_{t}\mu^n )$. Since this random variable is always bounded from above by $\diam(\T)^p$, by choosing e.g.\  $d=1/2$ in \eqref{eq:polynomial} of 
\cref{prop_Bernstein} below, we see that, if we pick $C_1 = (\ln a)^{-1} K$ sufficiently large -- precisely such that $5-Kd^2<p/2$, we can safely reduce ourselves to argue on the event $\|\rho_{t,n}-1\|\le 1/2$, so that $P_t \mu^n = \rho_{t,n}\mm$ has a density uniformly bounded from below by $1/2$. On such event, we use \eqref{eq:peyre-riesz} (with $\mu = P_{1/n} \mu^n$ and $\nu = P_{t} \mu^n$),  and we find
\begin{equation}\label{eq:good-bound}
 W_p^p(P_{1/n}\mu^n, P_{t}\mu^n ) \le C \int_{\T} \left| (-\Delta)^{-1/2} ( \rho_{1/n,n} - \rho_{t,n}) \right|^p \de \mm.
\end{equation}
where $C = C(\T, p)<\infty$. By the linearity of the operator $(-\Delta)^{-1/2}$, we collect the identity
$$  (-\Delta)^{-1/2}	 ( \rho_{1/n,n} - \rho_{t,n}) (x)= \frac 1 n \sum_{i=1}^n  [(-\Delta)^{-1/2} (p_{1/n}-p_{t})] (X_i- x),
$$ and notice that, for each $x\in \T$, the random variables $\varphi_i(x):= [(-\Delta)^{-1/2} (p_{1/n}-p_{t})] (X_i- x)$, for $i=1,\ldots, n$ are independent and centered. After taking expectation in \eqref{eq:good-bound}, we see that the thesis amounts to bound from above the quantity
$$ \int_{\T} \E{\left| \frac 1 n \sum_{i=1}^n \varphi_i(x) \right|^p }\de \mm(x),
$$
where we recognize, for every $x$, the $p$-th moment of a sum of independent centered random variables. By Rosenthal's inequality, \cite{rosenthal1970subspaces}, we have for some constant $C = C(p)<\infty$,
\begin{equation}\label{eq:rosenthal}
 \E{\left| \frac 1 n \sum_{i=1}^n \varphi_i(x) \right|^p } \le C \left[ \frac{1}{n^{p-1}} \E{\left| \varphi(x) \right|^p } + \frac{1}{n^{p/2} } \E{\left| \varphi(x) \right|^2 }^{p/2} \right],
\end{equation}
where we write $\varphi :=  [(-\Delta)^{-1/2} (p_{1/n}-p_{t})] (X_1- x)$.
 To conclude, we follow very closely the argument in \cite{Le17} from eq. (34) onwards (in the case $d=2$), so we omit some details. We collect first the uniform bound, valid for $0 < s \le 1/2$:
$$ \sup_{z \in \T } \left|  (-\Delta)^{-1/2} (p_{s} - 1) (z) \right| \le  \frac{C}{s^{1/2}},$$
which we apply in particular to $s \in \{1/n, t\}$, yielding
$$ \sup_{z \in \T } \left| \varphi (z) \right|  \le \sup_{z \in \T } \left|  (-\Delta)^{-1/2} (p_{1/n} - 1) (z) \right| + \sup_{z \in \T } \left|  (-\Delta)^{-1/2} (p_{t} - 1) (z) \right|  \le Cn^{1/2}.$$
Then,  by the representation $(-\Delta)^{-1} = \int_0^\infty P_s ds$, we find for any $f \in L^2(\T)$ with $\int_{\T} f \de\mm = 0$ that
\begin{equation*} \begin{split}
                   \int_{\T} [(-\Delta)^{-1/2}f]^2 \de \mm & = \int_{\T} f (-\Delta)^{-1} f \de \mm \\
                   & = \int_0^\infty \int_\T f P_sf \de \mm \de s= \int_0^\infty \int_\T (P_{s/2}f)^2 \de \mm \de s.
                  \end{split}
\end{equation*}
We use this identity in our case, i.e., with $f = p_{1/n}-p_t$, yielding
\begin{equation*}\begin{split}
 \E{\left| \varphi(x) \right|^2 }  & = \int_{\T} [ (-\Delta)^{-1/2}(p_{1/n} - p_t)]^2(y-x)  \de \mm(y)\\
& = \int_0^\infty \int_{\T} (p_{s/2+1/n}(y-x) - p_{s/2+t}(y-x) )^2 \de \mm(y) \de s\\
& = \int_0^\infty [ p_{s+2/n}(0) + p_{s+2t}(0) - 2 p_{s+t+1/n}(0)] \de s \\
& = O\left( -\log(2/n) - \log(2t) +2 \log(t+1/n) + 1 \right) = O( \ln \alpha),
\end{split}
\end{equation*}
where in developing the square we invoked the semigroup property (so that, for any $t_1, t_2 >0$, $\int_\T p_{t_1}(y-x) p_{t_2}(y-x) \de\mm(y)=P_{t_1+t_2} \delta_x(x)=p_{t_1+t_2}(0)$), and the final asymptotics can be computed directly from (\ref{def p_t}).

Combining these bounds, we find
$$ \E{\left|\varphi(x) \right|^p } \le \sup_{z\in \T}\left|\varphi(z) \right|^{p-2}  \E{\left| \varphi(x) \right|^2 }\\
 \le C n^{(p-2)/2} \ln \alpha ,$$
and therefore we bound from above the right hand side in \eqref{eq:rosenthal} with
\begin{equation}
\left[ \frac{1}{n^{p-1}} \E{\left| \varphi(x) \right|^p } + \frac{1}{n^{p/2} } \E{\left| \varphi(x) \right|^2 }^{p/2} \right] \le C \frac{\ln \alpha}{n^{p/2}} + C \left( \frac{\ln \alpha}{n} \right)^{p/2}
\end{equation}
and the thesis follows.
\end{proof}

In the proof above we used a regularizing property of the heat semigroup, when acting on empirical measures, as established in \cite{AG} (see Theorem~3.3 and Remark 3.17 therein), that we report here.

\begin{proposition}\label{prop_Bernstein}
If $\mu^n$ are as in \eqref{eq:empirical} and $P_t\mu^n=\rho_{t,n}\mm$, then
$$
\P{\{\|\rho_{t,n}-1\|_\infty>d\}}\leq \frac{C_3(\T)}{d^2t^3}a^{-nt d^2}
$$
for some $C_3(\T)>0$ and $a=a(\T)>1$. In particular, if $d\geq n^{-1}$ and $t=(\ln a)^{-1}K n^{-1}\ln n$ with $K\geq 1$, then
\begin{equation}\label{eq:polynomial}
\P{\{\|\rho_{t,n}-1\|_\infty>d\}}\leq C_3(\T)(\ln a)^3 n^{5-Kd^2}.
\end{equation}
\end{proposition}

Our strategy for proving \cref{tmain} will be to adjust the parameters $K=K_n\to\infty$ and $d=d_n\to 0$ in such a way that $Kd$ is sufficiently
large, so that the probability of the deviation from the constant density $1$ will have the power like decay we need with respect to $n$.

We will also need $L^p$ estimates on $\rho_{t,n}$, provided by the following proposition.

\begin{proposition}\label{prop_Lp}
Let $t_n$ as in \cref{tmain}, and $K=K_n$ related to $t_n$ as in \cref{prop_Bernstein}. Fixed $k>0$, take $c_n \rightarrow 0^+$ such that
\begin{equation}\label{cn}
\liminf_n K_nc_n^2 > k+5.
\end{equation}
Then
$$
\sup\left\{n^k\E{ \mathbf{1}_{\{\|\rho_{t_n,n}-1\|_\infty >c_n\}} \int_{\T}|\rho_{t_n,n}-1|^p\de\mm}:\,\, n\geq 2\right\}<\infty.
$$
\end{proposition}
\begin{proof} In this proof $C$ denotes a positive constant, depending only on $\T$.
Arguing as in the proof of Theorem~3.3 in \cite{AG}, the bounds 
$$
\E{Y^2}\leq \frac{C}{t},\qquad |Y|\leq\frac{C}{t}\qquad t\in (0,1)
$$
for the random variables $Y=Y_i=p_t(X_i,y)-1$, together with Bernstein's inequality yield
$$
\P{\{|\rho_{t,n}(y)-1|>\xi\}}\leq C\exp(-nct\xi)\qquad\forall t\in (0,1),\,\,\xi>1
$$
for all $y\in\T$. 
For our choice of $t=t_n$, Fubini's theorem and Cavalieri's formula yield
$$
n^k\E{\int_{ \{|\rho_{t_n,n}-1|>1\} } |\rho_{t_n,n}-1|^p\de\mm}\leq
C\int_1^\infty n^{k-c(\ln a)^{-1}K_n\xi} \cdot \xi^{p-1}\de\xi.
$$
Thus, for $n \gg 1$
\begin{align*}
&n^k\E{\int_{ \{|\rho_{t_n,n}-1|>1\} }|\rho_{t_n,n}-1|^p\de\mm} \leq
C\int_ 1^\infty n^{\xi (k-c(\ln a)^{-1}K_n)} \xi^{p-1} d\xi \\
&\leq
C\int_1^\infty 2^{-\xi}\xi^{p-1}\de\xi < \infty.
\end{align*}
On the other hand, exploiting \cref{prop_Bernstein} along with (\ref{cn}) we get
\begin{align*}
&n^k\E{ \mathbf{1}_{\{\|\rho_{t_n,n}-1\|_\infty >c_n \}} \int_{\{|\rho_{t_n,n}-1|\leq 1\}}|\rho_{t_n,n}-1|^p\de\mm} \leq n^k \P{\{\|\rho_{t_n,n}-1\|_\infty >c_n\}} \\
&\leq C_3(\T)(\ln a)^3 n^{k+5-K_nc_n^2} \rightarrow 0. \qedhere
\end{align*}
\end{proof}

\section{Propagation of $q$-laplacian estimates and differentiation of $\int |\nabla Q_t\phi|^q\s\de\mm$}

Recalling the definition of $c\geq 0$ in \cref{ULB}, $\phi\in H^{1,q}(\T)$ satisfies in a distributional sense the inequality
\begin{equation}\label{divc} -{\rm div}\bigl(|\nabla\phi|^{q-2}\nabla\phi)+ c\geq 0.\end{equation}
Namely, for every nonnegative $\eta \in C^\infty(\T)$ we have
\begin{equation}\label{distr}  \int_\T |\nabla\phi|^{q-2}\langle \nabla\phi, \nabla \eta \rangle\s \de\mm+c \int_\T \eta \s \de\mm\geq 0. \end{equation}
In order to control the time derivative of $\int_\T |\nabla Q_t\phi|^q \de\mm$, we would like to show that (\ref{distr}) propagates with the Hopf-Lax semigroup, that is, it is satisfied also by $Q_t\phi$ for any $t \in (0,1)$. The proof of this stability property becomes much easier if we understand \eqref{divc} in the viscosity sense; this is possible thanks
to the following result (see \cite{JJ}, \cite{JLM} for the homogeneous case $g=0$ and \cref{easy_impl} below):

\begin{theorem} \label{thm_JUU} Let $f\in H^{1,q}(\T)$ and $g:\T\to\R$ be continuous. Then $-\Delta_q f+g\geq 0$ in the viscosity sense, according to \cref{DefV}, if and only if  $-\Delta_q f+g\geq 0$ in the sense of distributions.
\end{theorem}

We are going to use \cref{thm_JUU} both ways: first we pass from the distributional sense for $\phi$, granted by \eqref{Lin-PDE}, to the viscosity
sense, then we pass from the viscosity sense to the distributional sense for $Q_s\phi$ in the proof of \cref{fine_Lambda}.

Then, let us show propagation of the estimate $-\Delta_q \phi +c\geq 0$ to $Q_t\phi$ in the viscosity sense. Actually it will be usefult to prove this property for 
the Hopf-Lax semigroup associated to any power $r>1$. We provide a direct proof, even though the statement could directly follow by the general fact that
viscosity supersolutions to $-\Delta_q +c\geq 0$ are stable under translations in the dependent and independent variables, and infimum.

\begin{proposition}\label{Prop-viscos}
Let $f:\T\to\R$ be lower semicontinuous and satisfying $-\Delta_q f+c\geq 0$ in the viscosity sense and $r\in (1,\infty)$. Then for all $t>0$ the function 
$$
f_t(x):=\min_{y \in \T} \left\{ f(y)+\frac{{\mathbf d}^r(x,y)}{rt^{r-1}} \right\}.
$$
still satisfies $-\Delta_q f_t+c\geq 0$ in the viscosity sense.
\end{proposition}

\begin{proof} Given $t>0$ and $x_0\in\T$, let $y_0\in\T$ be a point where the minimum in the definition of $f_t$ is attained, so that
$$f_t(x_0)=f(y_0)+\frac{{\mathbf d}^r(x_0,y_0)}{rt^{r-1}}.$$ 
Consider $\varphi\in C^2(\T)$ such that $f_t-\varphi$ has a local minimum in $x_0$ and, with no loss of generality, assume that the minimum is global and
$f_t(x_0)=\varphi(x_0)$. \\
If we set $\psi(x):=\varphi(x-y_0+x_0)$, we claim that $\phi-\psi$ has a minimum in $y_0$, equal to $-{\mathbf d}^r(x_0,y_0)/(rt^{r-1})$. From this we would obtain
					$$
					F_q(\nabla\psi(y_0),\nabla^2\psi(y_0))\leq c
					$$
					thus
					$$
					F_q(\nabla\varphi(x_0),\nabla^2\varphi(x_0))\leq c.
					$$
To prove the claim, we notice that
					$$
					\phi(y_0)-\psi(y_0)=\phi(y_0)-\varphi(x_0)=\phi(y_0)-f_t(x_0)
					=-\frac{1}{rt^{r-1}}{\mathbf d}^r(x_0,y_0),
					$$
					while in the other hand $f_t(x)\ge\varphi(x)$ implies
					$$
					\phi(y)+\frac{1}{rt^{r-1}}{\mathbf d}^r(x,y) \geq\varphi(x)\quad\forall x,\,y.
					$$
					Choosing $y=x-x_0+y_0$ (understanding the sum modulo $\Z^2$) we obtain
					$$
					\phi(y)-\psi(y)\geq - \frac{1}{rt^{r-1}}{\mathbf d}^r(x_0,y_0) \quad\forall y,
					$$
					as desired.
					\end{proof}
			
\begin{remark}\label{easy_impl} We can use \cref{Prop-viscos} to provide a sketchy proof of the implication from viscous to distributional
granted, also in the converse direction, by \cref{thm_JUU}. Indeed, we can use the Hopf-Lax semigroup with power $r=2$ to
obtain that $f_s=Q_sf$ still satisfy $-\Delta_q f_s+c\geq 0$ in the viscosity sense and $C^{1,1}$ regularity of $f_s$. Since $f_s\to f$ in $H^{1,q}(\T)$ as $s\to 0^+$, it is then sufficient to show
that $-\Delta_q f_s+c\geq 0$ in the sense of distributions. Here we can use the $C^{1,1}$ regularity of $f_s$ to build appropriate test functions $\phi$, of the form
$$
\phi(x)=f_s(x_0)+\langle\nabla f_s(x_0),x-x_0\rangle+\frac 12\langle\nabla^2 f_s(x_0)(x-x_0),(x-x_0)\rangle-\epsilon|x-x_0|^2
$$
at any point $x_0\in\T$ where $\nabla f_s(x_0)\neq 0$ and $\nabla^2f_s(x_0)$ exists. This leads to the validity of $-\Delta_q f_s+c\geq 0$ almost everywhere
in the open set $\Omega_s=\{|\nabla f_s|\neq 0\}$. Then, one obtains the validity of the inequality in the sense of distributions first in $\Omega_s$ and then on the whole of $\T$,
using the fact that the flux of the continuous vector field $|\nabla f_s|^{q-2}\nabla f_s$ is null on the boundary (because $q>1$). If $\Omega_s$ is not
smooth one can perform a further approximation, since
$$
\frac{1}{\epsilon}\int_0^\epsilon\int_{\{|\nabla f_s|=\tau\}}|\nabla f_s|^{q-1}\de \Haus^1\de\tau=
\int_{\{0<|\nabla f_s|<\epsilon\}}|\nabla f_s|^q\de\mm
$$
tends to $0$ as $\epsilon\to 0$.
\end{remark}

Now, we apply \cref{Prop-viscos} with $f=\phi$ and $r=p$ in order to estimate the variation in time of  $\int_\T |\nabla Q_t\phi|^q \de\mm$.

\begin{lemma}\label{fine_Lambda}
Let $\Lambda(t):=\int_\T |\nabla Q_t\phi|^q \de\mm$ with $\phi$ as in \eqref{ULB} and $c=\|\rho_1-\rho_0\|_\infty$. Then $\Lambda$ is Lipschitz in $[0,1]$ 
and $\frac{\de}{\de t} \Lambda(t) \leq c\Lambda(t)$ for almost every $t\in (0,1)$. In particular
\begin{equation}\label{diff-int}
\int_\T |\nabla Q_t\phi|^q \de\mm \leq e^{ct} \int_\T |\nabla\phi|^q \de \mm\qquad\forall t\in [0,1].
\end{equation}
\end{lemma}

\begin{proof} Thanks to \cref{Prop-viscos}, $f_t=Q_t\phi$ satisfy $-\Delta_q f_t+c\geq 0$ in the viscosity sense. Therefore \cref{thm_JUU}
grants this property also in the sense of distributions, namely (notice that the improvement from $C^\infty(\T)$ to $H^{1,q}(\T)$ follows by density and
$L^p$ integrability of $|\nabla f_t|^{q-2}\nabla f_t$)
\begin{equation}\label{propag} \int_\T |\nabla f_t|^{q-2}\langle \nabla f_t, \nabla \eta \rangle\s \de\mm+c \int_\T \eta \s \de\mm\geq 0 
\qquad\forall\eta \in H^{1,q}(\T), \eta \geq 0. \end{equation}

First, we note that, by \cref{propag}, the distribution $T:=-\div(|\nabla f_t|^{q-2} \nabla f_t)+c$ is non-negative. Thus, if $\eta \in C^\infty(\T)$
$$\scalprod T\eta \leq \scalprod T{\|\eta\|_\infty 1}={\|\eta\|}_{\infty}c$$
and then $T$ is represented by a nonnegative finite measure with mass less or equal than $c$ (here we used that $\scalprod {\div(|\nabla f_t|^{q-2} \nabla f_t)}1=0$ and therefore $\scalprod T1=c$). 
It follows that $\mu_t:=\div(|\nabla f_t|^{q-2} \nabla f_t)$ is a signed measure with $\| \mu_t\| \leq 2c$. \\
By the convexity of $y \mapsto |y|^q$ we then infer
\begin{equation}\label{eq:rough}
\Lambda(t)-\Lambda(s) \geq q\int_\T |\nabla f_s|^{q-2} \scalprod{\nabla f_s}{\nabla(f_t-f_s)} \de\mm=q\int_\T (f_s-f_t) \de\mu_s\geq -2cq\|f_t-f_s\|_\infty
\end{equation}
for every $s,\,t \in [0,1]$.
From the Lipschitz regularity of the initial datum $\phi$ (which follows by \cite[Theorem 2.1]{Te}) and \cref{prop_HL}
we deduce that the map $t \mapsto f_t$ is Lipschitz with respect to the sup norm, let us say with constant $L$. 
Hence, exchanging the roles of $t$ and $s$ we conclude that
$$|\Lambda(t)-\Lambda(s)| \leq 2cqL|t-s|,$$
as we desired.

Now, we can refine \eqref{eq:rough} as follows.
Let $t\in (0,1)$ be a differentiability point for $\Lambda$ such that $-q\frac\de{\de t}f_t=|\nabla f_t|^q$ a.e. in $\T$. 
Thanks to Rademacher's theorem and \cref{prop_HL}, both properties are satisfied for a.e. $t\in (0,1)$.
For $s\geq t$, using the inequality $f_t\geq f_s$ granted directly from the definition \eqref{eq:defHL}, as well
as the inequality $-\div(|\nabla f_s|^{q-2}\nabla f_s) +c\geq 0$ in the sense of distributions, we get
\begin{eqnarray*}
\Lambda(s)-\Lambda(t) &\leq& -q\int_\T |\nabla f_s|^{q-2} \scalprod{\nabla f_s}{\nabla(f_t-f_s)} \de\mm \\
&=&q\int_\T \div (|\nabla f_s|^{q-2}\nabla f_s)(f_t-f_s)\de\mm\\
&\leq&cq\int_\T (f_t-f_s) \de\mm, 
\end{eqnarray*}
so that
$$\frac \de{\de t} \Lambda(t)=\lim_{s \rightarrow t^+} \frac{\Lambda(s)-\Lambda(t)}{s-t} \leq \lim_{s \rightarrow t^+} cq \int_\T\frac{ f_t-f_s}{s-t} \de\mm=-cq\int_\T \frac\de{\de t} f_t \de\mm=c\int_\T |\nabla f_t|^q \de\mm,$$
which proves that $\Lambda'(t)\leq c\Lambda(t)$. Finally, the validity of \eqref{diff-int} follows by Gronwall's Lemma. \qedhere
\end{proof}

\section{Proof of \cref{ULB}}

In this section $\phi$, $c$ are as in \cref{ULB}.

\subsection{Upper bound}\label{UB}

The upper bound in \cref{ULB} can be obtained immediately by repeating the argument in Proposition~2.3 of \cite{AST}, 
involving duality and the Hopf-Lax formula. We still give the proof here for the sake of completeness. 

Since $\T$ is compact, the duality formula for $W_p^p$ can be written in the form
\begin{equation}\label{duality}
\frac 1p W_p^p(\rho_0 \mm, \rho_1 \mm)=\sup\left\{ -\int_\T f \rho_0 \de\mm+\int_\T (Q_1 f) \rho_1 \de\mm: 
\text{$f:\T\to\R$ Lipschitz} \right\}.
\end{equation}

\begin{proposition}[Upper bound]\label{prop:upper-bound}
There exists $\overline{\delta}(c,p)$ such that $\overline{\delta}(c,p) \rightarrow 0$ as $c\to 0$ and
$$W_p^p(\rho_0 \mm, \rho_1 \mm) \leq (1+\overline{\delta}(c,p)) \int_\T |\nabla \phi|^q \de\mm.$$
\end{proposition}

\begin{proof}
Let us bound uniformly the argument of the supremum in (\ref{duality}), for $f:\T\to\R$ Lipschitz, 
exploiting the PDE (\ref{q-Poisson}) satisfied by $\phi$, the fact that $Q_t f$ solves (\ref{HJ}) almost everywhere
in $(0,1)\times\T$ and dominated convergence to put $\frac \de{\de s}$ under the integral sign. If we set $\rho_t:=t\rho_1+(1-t)\rho_0$ per $t \in (0,1)$, then:
\begin{align}\label{HL}
&\int_\T (\rho_1 Q_1f-\rho_0f ) \de\mm=\int_0^1 \frac \de{\de s} \int_\T \rho_s Q_sf \de\mm \de s \nonumber \\
&=\int_0^1\int_\T \left(\rho_s \frac \de{\de s} Q_sf + (\rho_1-\rho_0)Q_sf \right) \de\mm \de s \nonumber \\
&=\int_0^1\int_\T \left( -\frac1q |\nabla Q_s f|^q \rho_s+|\nabla \phi|^{q-2} \langle \nabla \phi, \nabla Q_s f \rangle \right) \de \mm \de s \nonumber \\
&\leq \int_0^1\int_\T \left( -\frac1q |\nabla \phi|^q \rho_s^{-\frac q{q-1}}\rho_s+|\nabla \phi|^q \rho_s^{-\frac1{q-1}} \right) \de \mm \de s \nonumber \\
& =\frac 1p \int_\T \left( \int_0^1 \rho_s^{-\frac 1{q-1}} \de s \right) |\nabla \phi|^q \de\mm,
\end{align}
where for the inequality we used that $v=\rho_s^{-\frac 1{q-1}}\nabla \phi $ minimizes $v \mapsto \frac 1q |v|^q \rho_s-|\nabla \phi|^{q-2} \langle \nabla \phi, v \rangle$. \\
In conclusion
$$
W_p^p(\rho_0 \mm, \rho_1 \mm) \leq \int_\T M_q(\rho_0, \rho_1) |\nabla \phi|^q \de\mm,
$$
where $M_q(\rho_0, \rho_1)(x)=\int_0^1 \rho_s(x)^{-\frac 1{q-1}} \de s \lesssim 1$ as $c\to 0$. More precisely, since
$\|\rho_i-1\|_\infty\leq c/2$, for $c<2$ one has
$$M_q(\rho_0, \rho_1)(x) \leq  1+\overline{\delta}(c,p)  $$
with $\overline{\delta}(c,p)=(1-c/2)^{-\frac 1{q-1}}- 1$.\qedhere
\end{proof}

\subsection{Lower bound}\label{LB}


\begin{proposition}[Lower bound]
There exists $\underline{\delta}(c,p)$ such that $\underline{\delta}(c,p) \rightarrow 0$ as $c\to 0$ and
$$W_p^p(\rho_0 \mm, \rho_1 \mm) \geq (1-\underline{\delta}(c,p)) \int_\T |\nabla \phi|^q \de\mm.$$
\end{proposition}

\begin{proof} From the duality formula, with an integration by parts and Fubini's theorem, we get
\begin{eqnarray*}
\frac 1p W_p^p(\rho_1\mm,\rho_0\mm)&\geq& -\int_\T\phi\rho_0\de\mm+\int_\T (Q_1\phi)\rho_1\de\mm=
\int_\T\phi(\rho_1-\rho_0)\de\mm+\int_\T (Q_1\phi-\phi)\rho_1\de\mm\\
&=&\int_\T|\nabla\phi|^q\de\mm-\frac 1q\int_0^1\int_\T|\nabla Q_s\phi|^q\rho_1\de\mm\de s.
\end{eqnarray*}
Now we can use first the inequality $\|\rho_1-1\|_\infty\leq c/2$ to replace $\rho_1$ with $1$ and then \cref{fine_Lambda} with
$c\geq\|\rho_1-\rho_0\|_\infty$ to
estimate
$$
\frac 1p W_p^p(\rho_1\mm,\rho_0\mm)\geq \frac 1 p\int_\T|\nabla\phi|^q\de\mm-
\frac{c e^{c}/2+e^{c}-1}{q}\int_\T|\nabla\phi|^q\de\mm,
$$
so that $\underline{\delta}(c,p)=(p-1)(c e^{c}/2+e^{c}-1)$.\qedhere
\end{proof}

\section{Proof of \cref{tmain}}

In this section we adopt the notation in the statement of \cref{tmain}.


\subsection{Upper bound}

Since $\ln (nt_n)\ll\ln n$, using \cref{prop:dispersive} and the triangle inequality for $W_p$, arguing as in \cite{AST}, the proof of the upper bound reduces to
the following estimate:
\begin{equation}\label{eq_ubfinal1}
\limsup_{n\to\infty}\biggl(\frac{n}{\ln n}\biggr)^{p/2}\biggl(\E{W_p^p(P_{t_n}\mu^n,P_{t_n}\nu^n)}-\E{\int_\T|\nabla\phi|^q\de\mm}\biggr)\leq 0
\end{equation}
where $\phi$ is the solution to \eqref{q-Poisson} with right hand side
$$
\rho_{0,n}\mm=P_{t_n}\mu^n,\qquad\rho_{1,n}\mm=P_{t_n}\nu^n.
$$ 
Now, since $t_n\gg n^{-1}\ln n$ we can use \cref{prop_Bernstein} to write $t_n$ as $(\ln a)^{-1}K_n n^{-1}\ln n$ with $K_n\geq 1$
and $c_n\to 0$ in such a way that $K_n c_n^2>2p+10$, so that
$$
\P{\{\|\rho_{i,n}-1\|_\infty>\frac {c_n} 2\}}\leq C_3(\T)(\ln a)^3 n^{5-K_nc_n^2/4}=O(n^{-p/2})\qquad i=0,1.
$$
Since $W_p(\mu,\nu)\leq {\rm diam}(\T)$ for any pair of probability measures $\mu$, $\nu$,
it follows that the contribution to \eqref{eq_ubfinal1} of the event $\{\max_i\|\rho_{i,n}-1\|_\infty>\frac {c_n} 2\}$ is null and in the complementary
event we can use \cref{ULB} to conclude.

\subsection{Lower bound}

Recall that the semigroup $P_t$ is contractive in $\T$ with respect to any $W_p$ distance; this can be easily proved taking any
coupling $\Sigma$ between $\mu$ and $\nu$ and considering the average $\bar\Sigma=\int\Sigma_zp_t(z)\de\mm(z)$ of the shifted couplings
$$
\Sigma_z:=(\tau_z\times\tau_z)_\#\Sigma\qquad\text{with}\qquad \tau_z(x)=x+z,
$$
which provides a coupling between $P_t\mu$ and $P_t\nu$ with the same cost.
Therefore, the lower bound 
\begin{equation}\label{eq_lbfinal}
\liminf_{n\to\infty}\biggl(\frac{n}{\ln n}\biggr)^{p/2}\biggl(\E{W_p^p(\mu^n,\nu^n)}-\E{\int_\T|\nabla\phi|^q\de\mm}\biggr)\geq 0
\end{equation}
can be deduced from
\begin{equation}\label{eq_lbfinal1}
\liminf_{n\to\infty}\biggl(\frac{n}{\ln n}\biggr)^{p/2}\biggl(\E{W_p^p(\rho_{0,n}\mm,\rho_{1,n})\mm}-\E{\int_\T|\nabla\phi|^q\de\mm}\biggr)\geq 0.
\end{equation}
Now, recall that the solution $\phi$ to \eqref{Lin-PDE} is the unique minimizer of the functional
$$
\Lambda_q(f):=\int_\T\frac 1q|\nabla f|^q-f(\rho_1-\rho_0)\de\mm=\int_\T\frac 1q|\nabla f|^q-(f-\bar f)(\rho_1-\rho_0)\de\mm
$$
whose minimum value is nonpositive. Hence, from the Sobolev embedding we obtain 
$$
\frac 1 q\int_\T|\nabla\phi|^q\de\mm\leq\|\rho_1-\rho_0\|_p\|\phi\|_q\leq c_S\|\rho_1-\rho_0\|_p\biggl(\int_\T|\nabla\phi|^q\de\mm\biggr)^{1/q}
$$
and then the deterministic upper bound
\begin{equation}\label{eq:ubdirq}
\int_\T|\nabla\phi|^q\de\mm\leq (\|\rho_1-\rho_0\|_p c_S q)^p.
\end{equation}
 As in the proof of the upper bound, since $nt_n\gg\ln n$, we can use this time \cref{prop_Lp} that provides
an estimate in expectation on $\|\rho_{i,n}-1\|_p^p$ to show that the contribution to \eqref{eq_lbfinal1} of the event $\{\max_i\|\rho_{i,n}-1\|_\infty>\frac {c_n} 2\}$ is null (if we also require $K_nc_n^2>2p+20$ in order to satisfy (\ref{cn}) with $\frac{c_n}2$ and $k=\frac p2$) and in the complementary event we can use Theorem~\ref{ULB} to conclude.

%
%

\printbibliography

\end{document}

\end{document}